\documentclass[12pt,a4paper]{amsart}
\usepackage{latexsym,amsfonts,amsthm,amsmath,mathrsfs,amssymb}
\usepackage[british]{babel}
\usepackage[dvips]{color}
\usepackage[utf8]{inputenc}
\usepackage[T1]{fontenc}
\usepackage[dvips,final]{graphics}
\usepackage{xcolor}
\usepackage{mathrsfs}
\usepackage{mathtools} 
\usepackage{breqn}
\usepackage{epstopdf}
\usepackage{verbatim}
\usepackage{amscd}
\usepackage{hyperref}
\usepackage[english,capitalize]{cleveref}
\usepackage{aliascnt}

\makeatletter
\def\newaliasedtheorem#1[#2]#3{
	\newaliascnt{#1@alt}{#2}
	\newtheorem{#1}[#1@alt]{#3}
	\expandafter\newcommand\csname #1@altname\endcsname{#3}
}
\makeatother

\theoremstyle{plain}

\newtheorem{theorem}{Theorem}[section]
\newaliasedtheorem{prop}[theorem]{Proposition}
\newaliasedtheorem{lem}[theorem]{Lemma}
\newaliasedtheorem{coroll}[theorem]{Corollary}

\theoremstyle{definition}
\newaliasedtheorem{defi}[theorem]{Definition}
\newaliasedtheorem{quest}[theorem]{Question}
\newaliasedtheorem{fact}[theorem]{Fact}

\theoremstyle{remark}
\newaliasedtheorem{rem}[theorem]{Remark}
\newaliasedtheorem{exa}[theorem]{Example}

\newcommand{\R}{\mathbb{R}}
\newcommand{\N}{\mathbb{N}}

\newcommand{\HH}{\mathbb{H}}

\newcommand{\G}{\mathbb{G}}

\let\altphi\phi
\let\phi\varphi
\let\varphi\altphi
\let\altphi\undefined

\newcommand{\average}{{\mathchoice {\kern1ex\vcenter{\hrule height.4pt
width 6pt
depth0pt} \kern-9.7pt} {\kern1ex\vcenter{\hrule height.4pt width 4.3pt
depth0pt}
\kern-7pt} {} {} }}

\address{\textsc{Daniela Di Donato}: 
Dipartimento di Ingegneria Industriale e Scienze Matematiche, Via Brecce Bianche,12 60131 Ancona , Universit\'a Politecnica delle Marche.}
\email{daniela.didonato@unitn.it}

\title{Intrinsically Lipschitz  graphs on semidirect products of groups}

\date{\today}

\author{ Daniela Di Donato }

\subjclass[]{ 
	53C17, 
26A16,  
51F30, 
54E35. 
}
\keywords{Lipschitz graphs, Lie group, metric group, left-invariant distance}

\begin{document}

\maketitle

 \begin{abstract}
 In the metric spaces, we give some equivalent conditions of intrinsically Lipschitz maps introduce by Franchi, Serapioni and Serra Cassano in subRiemannian Carnot groups. Unlike what happens in the Carnot groups, in our context intrinsic dilation do not exist but we can prove  the same results using the Lipschitz property of the projection maps.  \end{abstract}

%
%

  \section{Introduction}
The notion of intrinsically Lipschitz maps was introduced by Franchi, Serapioni and Serra Cassano \cite{FSSC, FSSC03, MR2032504} (see also  \cite{SC16, FS16}) in the context of Heisenberg groups and then  in the more general Carnot groups in order to give a good notion of rectifiable sets inside these particular metric spaces. This is because  Ambrosio and Kirchheim \cite{AmbrosioKirchheimRect} show that the classical definition using Lipschitz maps given by Federer  \cite{Federer} does not work in subRiemannian Carnot groups \cite{ABB, BLU, CDPT}.  

Recently, Le Donne and the author generalize the concept of intrinsically Lipschitz maps in metric spaces \cite{DDLD21}. The difference between the two approaches is that Franchi, Serapioni and Serra Cassano study the properties of intrinsically Lipschitz maps; while we study the ''sections'' or rather the properties of the graphs that are intrinsic Lipschitz.

In a similar way of Euclidean case, Franchi, Serapioni and Serra Cassano introduce a suitable definition of intrinsic cones which is deep different to Euclidean cones and then they say that a map $\phi$ is intrinsic Lipschitz if  for any $ p\in  \mbox{graph}(\phi) $ it is possible to consider an intrinsic cone $\mathcal{C}$ with vertex on $p$ such that
\begin{equation*}
\mathcal{C} \cap \mbox{graph}(\phi) = \emptyset.
\end{equation*}
Roughly speaking, in the new approach studied in  \cite{DDLD21} a section $\psi$ is such that $\mbox{graph}(\phi) =\psi (Y) \subset X$ where $X$ is a metric space and $Y$ is a topological space. We prove some relevant properties as the Ahlfors regularity, the Ascoli-Arzel\'a Theorem, the Extension theorem, etc. in the context of metric spaces. Following this idea, the author introduce other two natural definitions: intrinsically H\"older  sections \cite{D22.1} and intrinsically quasi-isometric sections  \cite{D22.2} in metric spaces.

The purpose of this note is to give some equivalent conditions of intrinsically Lipschitz maps in the context of metric groups. More precisely, the main results  are Proposition \ref{coroll28dic}, Theorem \ref{prop1codim} and Proposition \ref{propLip}. These results are proved by Franchi and Serapioni \cite{FS16} in the context of Carnot groups; they use the properties given by the intrinsic dilations structure that do not exist in metric groups.

In particular, the term  {\em metric group} means that we are considering a topological group equipped with a left-invariant distance that induces the topology. In particular, when considering a metric Lie group, the distance would induce the manifold topology.

We shall considering   groups  that have the structure of semidirect product of two   groups. That is we consider groups of the form 
$G= N \rtimes H$ where $N $ and $ H$ are two    groups and $H$ acts on $N$ by automorphisms. Equivalently, the subgroup $N$ is normal within $ N \rtimes H$, and $ N \cap H = \{1\}$.

Another difference between metric groups and more specific Carnot groups is that, in the first setting, the projection map $\pi_{N}: N \rtimes H \to N$ is {\em Lipschitz at $1$}, i.e.,
\begin{equation}\label{def:splitat1}
d(1, \pi_{N}(g) )\leq K d(1,g), \qquad \forall g\in G.
\end{equation}
On the other hand, if $G=N \rtimes H$ is a metric group this is not true (see Remark 6.2 in  \cite{DDLD21}) but this Lipschitz property of the projection gives some good properties in order to obtain the same statements in this more general case where the intrinsic dilations structure does not exist.

 {\bf Acknowledgements.}  Part of this research was done while the author was visiting prof. Le Donne at the University of Fribourg. The excellent work atmosphere is acknowledged.

  \section{Notation}
   \subsection{Intrinsic graphs}\label{note1}

Let $N \rtimes H$ be a semidirect product of groups.
Given a subset $E\subset N$ and a  map $\phi: E \subset N \to H$ we call the {\em intrinsic graphing map} of $\phi$ the map
$\Phi: E \subset N \to N \rtimes H$ defined as
\begin{equation}\label{Phi}
\Phi (n):= n \cdot \phi (n), \quad \forall n\in E. 
\end{equation}
Moreover, we call the set
\begin{equation*}\label{graph}
  \Gamma_ {\phi}:=\{ n\cdot \phi (n) \, |\, n\in E \} = \Phi(E),
\end{equation*}
the {\em intrinsic graph of $\phi$}, which in other words is the   graph  of the intrinsic graphing function $\Phi$.
 
 A subset $S\subset  N \rtimes H$ is called an {\em intrinsic graph}, or   an {\em intrinsic $(N, H)$-graph}, if the structure of semidirect product is not clear, if there is $\phi: E \subset N \to H$ such that $S= \Gamma_ {\phi}$.
Clearly, we have that $S=\Phi (E )$ is equivalent to $S=\Gamma_ {\phi}$. If $\phi: N \to H$ is defined on whole of $N$, we say that $S=\Gamma_ {\phi}$ is an {\em entire} intrinsic graph.

By uniqueness of the components along $N$ and $H$, if $S=\Gamma_ {\phi} $ then $\phi$ is uniquely determined among all functions from $N$ to $H$.  Indeed, the set $E$ equals $\pi_N(S)$ and for all $n\in N$ we have that $
\phi(n)=\pi_H(n).$

 
 \begin{prop}\label{intr graph left translation}
The concept of intrinsic graph is preserved by left translation: 
For every $q\in G$, a set $S\subseteq N \rtimes H$ is an intrinsic graph if and only if 
$qS $ is an intrinsic graph.
 More precisely, for each $q\in G$ and $\phi: E \subset N \to H$, if we consider the set 
 \begin{equation}\label{defintraslfunct0}
 E _q :=\{ n \in N\, :\, \pi_N (q^{-1}n)\in E  \}\end{equation}
 and the map
  $\phi_q : E _q  \to H$ defined as
\begin{equation}\label{defintraslfunct}
\phi_q (n):= (\pi_H (q^{-1}n))^{-1} \phi( \pi_N (q^{-1}n)), \quad \mbox{ for all } n \in E_q,
\end{equation}
then
\begin{equation*}\label{defintraslfunct.01}
L_q ( \Gamma_ {\phi}) = \Gamma_ {\phi_q }.
\end{equation*}

 \end{prop}

  \begin{proof}
Fix $q\in G$, then
\begin{equation*}
\begin{aligned}
 \Gamma_ {\phi_q } & = \{ n \phi_q (n) \, :\, n\in E_q  \} \\
 & =\{ n   (\pi_H (q^{-1}n))^{-1} \phi( \pi_N (q^{-1}n))  \, :\, n\in E_q  \}\\
  & =\{ n [ n^{-1}q \pi_N (q^{-1}n)] \phi( \pi_N (q^{-1}n))  \, :\, \pi_N (q^{-1}n) \in E  \}\\
  & = L_q ( \Gamma_ {\phi}),
\end{aligned}
\end{equation*}
as desired.
  \end{proof}
  
We observe that if $q \in \Gamma _\phi$ then $\phi_{q^{-1}} (1)= 1$ and, from  the continuity of the projections $\pi_N$ and $\pi_H$, it follows that the continuity of a function is preserved by translations. Precisely given $q\in G$ and $\phi : N \to H$, then the translated function $\phi_q$ is continuous in $n\in N$ if and only if the function $\phi$ is continuous in the corresponding point $\pi_N (q^{-1}n)$.  
Moreover, for any $p,q \in G$  it follows that
\begin{equation*}
(\phi_p)_q= \phi_{q\cdot p}
\end{equation*}
indeed, by Proposition \ref{intr graph left translation},  $\Gamma_ {(\phi_p)_q }=L_q (\Gamma_ {\phi_p })=L_q(L_p (\Gamma_ {\phi}))=L_{q\cdot p} (\Gamma_ {\phi}).$ Consequently, $(\phi_p)_{p ^{-1}}= \phi_{p ^{-1}\cdot p}= \phi.$

 \begin{rem} Let $ (G=N \rtimes H,d)$ be a metric  group and let $\phi:N \to H$ be a continuous map. Then,
 \begin{equation*}
 {\rm dist}(p, \Gamma_ {\phi}) \leq  d(1,  \pi_H (p)^{-1} \phi( \pi_N (p))), \quad \forall p \in G,
\end{equation*} 
  where  ${\rm dist}(p, \Gamma_ {\phi}):=\inf\{ d(p, q)\, :\, q  \in \Gamma_ {\phi}\}.$ This follows by left invariance of $d$; indeed, for any $p \in G$ we have that 
  \begin{equation*}
\begin{aligned}
 {\rm dist}(p, \Gamma_ {\phi}) \leq d(p, \pi_N (p) \phi ( \pi_N (p))) = d( \pi_H (p), \phi ( \pi_N (p))) =  d(1,  \pi_H (p)^{-1} \phi( \pi_N (p))).
\end{aligned}
\end{equation*}
\end{rem}

    \subsection{Intrinsically Lipschitz maps: History}
    
Regarding Carnot groups, different notions of rectifiability have been proposed in the literature:
\begin{enumerate}
\item Rectifiability using images of Lipschitz maps defined on subsets of $\R^d$;
\item  Lipschitz image  rectifiability, using homogeneous subgroups;
\item  Intrinsic Lipschitz graphs  rectifiability;
\item  Rectifiability using intrinsic $C^1$ surfaces.
\end{enumerate}

The first approach (1) is a general metric space approach, given by Federer in \cite{Federer}. He states that a $d$-dimensional rectifiable set in a Carnot group $\G$ is essentially covered by the images of Lipschitz maps from $\R^d$ to a Carnot group $\G$. Unfortunately, this definition is too restrictive because often there are only rectifiable sets of measure zero (see \cite{AmbrosioKirchheimRect, biblioMAGN1}).

Another metric space approach but more fruitful than $(1)$ in the setting of groups is given by Pauls \cite{MR2048183} (see (2)). It is called Lipschitz image (LI) rectifiability. Pauls considers images in $\G$ of Lipschitz maps defined not on $\R^d$ but on subset of homogeneous subgroups of $\G.$

 Intrinsic Lipschitz graphs (iLG) rectifiability $(3)$ and the notion of intrinsic $C^1$ surfaces $(4)$ were both introduced by Franchi, Serapioni, Serra Cassano.  
In this paper we focus our attemption on the concept $(3)$ which we will introduce in the next section.  Moreover, the notion $(4)$ adapting to groups De Giorgi's classical technique valid in Euclidean spaces to show that the boundary of a finite perimeter set can be seen as a countable union of $C^1$ regular surfaces. A set $S$ is a $d$-codimensional intrinsic $C^1$ surface $(4)$  if there exists a continuous function $f:\G \to \R^d$ such that,  locally, $$S=\{ p\in \G : f(p)=0\},$$ and the horizontal jacobian of $f$ has maximum rank, locally.
  
The approaches $(2)$ and $(3)$ are natural counterparts of the notions of rectifiability in Euclidean spaces, where their equivalence is trivial. Hence it is surprising that the connection between iLG and LI rectifiability is poorly understood already in Carnot groups of step 2. 

In \cite{ALD}, Antonelli and Le Donne prove that these two definitions are different in general; their example is for a Carnot group of step $3$.  
 The paper \cite{biblioDDFO} makes progress towards the implication iLGs are LI rectifiable in $\HH^n$. We proved that $C^{1,\alpha } $-surfaces are LI rectifiable, where $C^{1,\alpha } $-surfaces are intrinsic $C^1$ ones  whose horizontal normal is $\alpha$-H\"older continuous. 

    \subsection{Intrinsically Lipschitz maps: Definition} Let $ (G=N \rtimes H,d)$ be a metric  group.  For a map $\psi:N\to H$ we  say that $\psi$ is an 
{\em  intrinsically Lipschitz map in the FSSC sense}  if  exists $K>0$ such that
\begin{equation}\label{equation2312}
d(1, \pi_{H}(x^{-1}x') )\leq K d(1, \pi_{N}(x^{-1}x') ), \qquad \forall x,x'\in  \Gamma_\psi.
\end{equation}
Regarding the bibliography, the reader can read \cite{ASCV, ADDDLD, BCSC, BSC, BSCgraphs, biblio26, Corni19, CM20, DD19, DD20, FMS14, FSSC11, JNGV, Mag13, MV, Vittone20}.

 The idea of this paper is to generalize some properties proved in Carnot groups in metric groups using the additional hypothesis that the projection map $\pi_{N}: N \rtimes H \to N$ is Lipschitz at $1_G$ (see \eqref{def:splitat1}). In order to do this, we  conclude this section give some equivalent conditions of this fact.
 \begin{prop}[\cite{DDLD21}]\label{Defi splitting is locally Lipschitz} 
Let $ (G=N \rtimes H,d)$ be a metric  group. The following  conditions are equivalent:
 \begin{enumerate}
 \item there is $C_1>0$  such that  $\pi _H : N \rtimes H \to H$ is a  $C_1$-Lipschitz map, 
i.e., 
  \begin{equation*}\label{equ1}
d(\pi_H (g),\pi_H (p)) \leq C_1d(g,p), \quad \forall g,p \in G;
\end{equation*} 
 \item there is $C_2>0$  such that 
 \begin{equation*}\label{equ4}
d(1,\pi_H (g)) + d(1,\pi_N (g)) \leq  C_2d(1,g), \quad \forall g \in G;
\end{equation*}
\item there is $C_3>0$   such that $\pi_N$ is $C_3$-Lipschitz at $1$, i.e.,
 \begin{equation*}\label{equ2}
d(1,\pi_N (g)) \leq C_3d(1,g), \quad \forall g \in G;
\end{equation*}
\item there is $C_4>0$ such that 
 \begin{equation*}\label{equ3}
d(1,\pi_H (g)) \leq C_4d(1,g), \quad \forall g \in G;
\end{equation*} 
\item  there is $C_5>0$ such that 
 \begin{equation*}\label{equ.020}
 d(1, \pi_N (g)) \leq  C_5{\rm dist} (g^{-1}, H), \quad \forall g \in G;
\end{equation*} 
\end{enumerate} 
\end{prop}

  \section{Intrinsic cones}
  
    \subsection{Intrinsic cones}\label{Intrinsic cones29dic}
  In this section, we present two definitions of cone which generalize the ones given by Franchi, Serapioni and Serra Cassano in the context of Carnot groups. The reader can see \cite{SC16, FS16} and their references. In particular, Definition~\ref{FMSdefi2.3.1}  is more general than Definition~\ref{defiIntrCONE}  because it does not require that $H$ is a complemented subgroup. Proposition~\ref{FMSprop2.3.2NUOVO} states that the equivalence of these two definitions when $\pi_H $ is a Lipschitz map.

\begin{defi}[Intrinsic cone] \label{FMSdefi2.3.1}
Let $ (G,d)$ be a metric  group and let $H$ be a subgroup of $G.$ The cones $ X_H( \alpha )$ with axis $H$, vertex $1$, opening $\alpha \in [0,1]$ are defined as
\begin{equation*}
X_H(\alpha)=\{ g \in G \, :\, {\rm dist} (g^{-1},H) \leq \alpha d(1,g) \}.
\end{equation*}
where dist$(g,H):=\inf\{d(1,gq)\, :\, q\in H\}$. For any $p\in G,$ $p\cdot X_H( \alpha )$ is the cone with base $N,$ axis $H,$ vertex $p,$ opening  $\alpha .$
\end{defi}

\begin{rem}
Notice that $G= X_H(1 )$ and $X_H(0)=H.$
\end{rem}

\begin{defi}[Intrinsic cone]\label{defiIntrCONE}
Let $ (N \rtimes H,d)$ be a metric  group, $q\in N \rtimes H$ and $\alpha \geq 0$. We define the cones  $C_{N,H} (\alpha)$ with base $N,$ axis $H,$ vertex $1,$ opening  $\alpha $ as following
\begin{equation*}\label{intrcones}
 C_{N,H} ( \alpha):= \{p\in G \,:\, d(1, \pi_N (p) ) \leq \alpha  d(1, \pi_H (p))\},
\end{equation*}
and $p\cdot C_{N,H} ( \alpha)$ is the cone with base $N,$ axis $H,$ vertex $p,$ opening  $\alpha .$
\end{defi}

\begin{rem}\label{27remsuiconi}
Notice that $H=C_{N,H} ( 0),$ $N \rtimes H = \overline{ \cup _{\alpha > 0} C_{N,H} ( \alpha)}$ and $C_{N,H} ( \alpha _1) \subset C_{N,H} ( \alpha _2)$ for $\alpha _1<\alpha _2.$ 
\end{rem}

\begin{rem}
Let $p\in  C_{N,H} ( \alpha)$ and $k \in \N$ with $k\geq 2.$ Then $p^k \in C_{N,H} (k^2+k( \alpha -1)).$ Indeed, for $p=nh$ with $h\in H$ and $ n\in N$, an explicit computation gives that 
\begin{equation*}
\pi_H(p^k)=h^k \quad \mbox{and} \quad \pi_N(p^k)=n \prod_{j=1}^{k-1} C_{h^j}(n),
\end{equation*}
and, consequently,
\begin{equation*}
d(1,\pi_N(p^k)) \leq kd(1,n)+2\sum_{j=1}^{k-1} j d(1,h) \leq [k^2+k( \alpha -1)] d(1,h),
\end{equation*}
i.e., $p^k \in C_{N,H} (k^2+k( \alpha -1)),$ as wished.  
\end{rem}

Before to investigate regarding the equivalence between these two definitions we present a result  which we will use in Section  \ref{Some equivalent conditions of intrinsically Lipschitz maps}:
 \begin{prop}[\cite{DDLD21}] \label{corol8.8}
Let $G=N\cdot H$ be a metric group such that $\pi_{N} $ is $k$-Lipschitz at $1$.   Let $\psi:N \to H$, $n \in N$ and $p =n \phi(n).$ Then the following statements are equivalent:
 \begin{enumerate}
 \item $\phi$ is intrinsically $L$-Lipschitz at point $n \in N$ with respect to  $d$ and with constant $L>0;$
\item  for all $\hat L\geq (k+1)L,$ it holds $$p\cdot X_{H} (1/\hat L)  \cap \Gamma_\phi =  \emptyset,$$
where $p\cdot X_{H}( \alpha )$ is the cone with  axis $H,$ vertex $p,$ opening  $\alpha $ defined as the translation of 
\begin{equation*} 
X_{H}(\alpha) =\{ g \in G \, :\, {\rm dist} (1,gH) < \alpha d(1,g) \}
\end{equation*}
where dist$(1,gH):=\inf\{d(1,g q)\, :\, q\in H\}$. 
\end{enumerate} 
\end{prop}

%
%

Locally, the intrinsic cone $p\cdot C_{N,H} (\beta)$ is equivalent to  $p\cdot X_H(\alpha)$ when  $\pi _H$ is a  Lipschitz map:

\begin{prop}\label{FMSprop2.3.2NUOVO}
 Assume that $ (G= N \rtimes H,d)$ is a metric  group, $p\in G$ and $\pi _H : G \to H$ is a  $C$-Lipschitz map. Then, for any $0<\alpha _1 <\frac 1  {C+1}$ there is  $\beta _1>0$  
 such that locally
 \begin{equation*}
p\cdot X_H( \alpha _1) \subset p\cdot C_{N,H} (\beta _1),
 \end{equation*}
 and for any $0<\beta _2 < \frac 1 C$ there is $\alpha _2 \in (0,1)$ such that  locally
 \begin{equation*}
p\cdot C_{N,H} (\beta _2) \subset p\cdot X_H( \alpha _2).
 \end{equation*} 
\end{prop}

\begin{proof}
 It is enough to prove the claim with $p=1$ because of the left translation of the distance $d.$  
 
We prove the first inclusion. Let $g \in X_H(\alpha _1),$ i.e.,  $ \mbox{dist} (g^{-1},H) \leq \alpha _1 d(1,g).$ Using Proposition \ref{Defi splitting is locally Lipschitz} (5) and noting $C_5=C+1,$ we have that
 \begin{equation*}
\begin{aligned}
 d(1, \pi_N (g)) & \leq  C_5{\rm dist} (g^{-1}, H) \leq  \alpha _1 C_5d(1,g) \leq \alpha _1(C+1) (d(1,\pi_N (g)) +d(1, \pi_H (g))) .
\end{aligned}
\end{equation*}
Hence we can choose $\beta _1$ so that $\beta _1 \geq \frac{\alpha _1(C+1)}{1-\alpha _1(C+1)}.$ Consequently,  $g\in C_{N,H} (\beta _1),$ as desired. 

Now we prove the second inclusion. Let $g \in  C_{N,H} (\beta _2),$ i.e., $d(1, \pi _N(g)) \leq \beta _2d(1, \pi_H(g)).$ Then, by Proposition \ref{Defi splitting is locally Lipschitz} (4)
 \begin{equation*}
\begin{aligned}
\mbox{dist} (g^{-1},H) & \leq  d (g^{-1}, \pi_H (g^{-1})) = d(1,n ^{-1}) = d(1,n)\leq  \beta _2d(1,\pi_H(g)) \leq \beta _2Cd(1,g).
\end{aligned}
\end{equation*}
Hence, if we choose $\alpha _2 = \beta _2 C$, we obtain that $g \in X_H(\alpha_2)$ and the proof is complete.
\end{proof} 

A corollary of Proposition \ref{FMSprop2.3.2NUOVO} is the following result

 \begin{prop}\label{coroll28dic} Let $ (G=N \rtimes H,d)$ be a metric  group with $\pi_H :G \to H$ Lipschitz map.   Let $\phi:N \to H$, $m \in N$ and $p =m \phi(m).$ Then the following statements are equivalent:
 \begin{enumerate}
 \item $\phi$ is intrinsically $L$-Lipschitz at point $m \in N$ with respect to  $d$ and with constant $L>0;$
\item  there is $\alpha \in (0,1)$ such that  $$p\cdot C_{N,H} (\alpha ) \cap \Gamma_{\phi } =\{p\}.$$
\end{enumerate} 
\end{prop}

\begin{proof}
	It is enough to combine Proposition \ref{corol8.8}  and Proposition \ref{FMSprop2.3.2NUOVO}.  
\end{proof}

\subsection{Intrinsic right and left cones} 

Notice that 
\begin{equation*}
 G =N \rtimes H \mbox{ if and only if } G =H \ltimes N,
\end{equation*}
it is natural to consider left and right cones as in  \cite{ACM12} where the authors consider them in the context of Heisenberg groups. Here we introduce these cones and then we study some properties and their link. As in Definition \ref{defiIntrCONE}, the left cone is $$C^\ell_{N,H} ( \alpha) \equiv C_{N,H} ( \alpha)= \{p\in G=N \rtimes H\,:\, d(1, \pi_N (p) ) \leq \alpha  d(1, \pi_H (p))\}.$$  on the other hand, the right cone $C^r_{N,H} (\alpha)$ with base $N,$ axis $H,$ vertex $1,$ opening  $\alpha $ is defined as following
\begin{equation}\label{intrcones27dic}
 C^r_{N,H} ( \alpha):= \{p\in G=H \ltimes N \,:\, d(1, \tilde \pi_N (p) ) \leq \alpha  d(1, \tilde \pi_H (p))\},
\end{equation}
where $ \tilde \pi_N : H \ltimes N \to N$ and $ \tilde \pi_H : H \ltimes N \to H$ are the natural projections on $G$ considering the splitting $H \ltimes N.$
 The right cones with vertex $p\in G$ are defined by left translation, i.e., $p\cdot C^r_{N,H} ( \alpha)$ is the cone with base $N,$ axis $H,$ vertex $p,$ opening  $\alpha .$

The left and right cones are comparable in the following sense:
 \begin{prop}\label{leftright}
  Let $ (G=N \rtimes H,d)$ be a metric  group. For any $p\in G$ and $\alpha , \beta \geq 0,$ it holds 
  \begin{equation*}
  \begin{aligned}
p\cdot C^\ell_{N,H} ( \alpha) & \subset p\cdot C^r_{N,H} ( \alpha +2 ), \\
p\cdot C^r_{N,H} ( \beta )  &\subset p\cdot C^\ell_{N,H} ( \beta +2).
\end{aligned}
\end{equation*}

  \end{prop}

 \begin{proof} Pick $\alpha \geq 0.$ 
 By left translation invariant, it is sufficient to show that 
  \begin{equation}\label{coni_sindestri}
 C^\ell_{N,H} ( \alpha) \subset C^r_{N,H} ( \alpha +2 ) \subset  C^\ell_{N,H} ( \alpha +4).
\end{equation} 

We begin observing a simple property of the projections. Let $p\in G.$ By uniqueness of the components along $N$ and $H,$ we know that $p= nh\in  N \rtimes H$ with $n \in N$ and $h \in H.$ On the other hand, because $G =H \ltimes N$ we have that $p=\ell m$ with $m \in N$ and $\ell \in H.$  Hence,
\begin{equation*}
nh=\ell m,
\end{equation*}
and so, by uniqueness of the components along $N$ and $H,$ we deduce that $$nh= \pi_N(\ell m) \pi_H(\ell m)= \pi_N(\ell m \ell ^{-1}\ell ) \pi_H(\ell m \ell ^{-1} \ell )=C_\ell (m) \ell.$$

That means $h=\ell$ and $n=C_h (m).$

Now, we prove the first inclusion in \eqref{coni_sindestri}. Let $p \in G$ as above and such that $p \in C^\ell_{N,H} ( \alpha).$ Then, by definition of the left cone we have  $d(1,n) \leq \alpha d(1,h)$ and, consequently,
\begin{equation*}
\begin{aligned}
d(1, m) &= d (1,h^{-1}  C_h (m) h) \leq  d (1, C_h (m) ) + 2 d (1, h) = d (1, n ) + 2 d (1, h) \leq (\alpha + 2)  d (1, h),
\end{aligned}
\end{equation*}
i.e. $p\in C^r_{N,H} ( \alpha +2 ),$ as desired. In a similar way, it is possible to show the second inclusion in \eqref{coni_sindestri}.

  \end{proof}

 \begin{rem}
 We underline that the projections in \eqref{intrcones27dic} are different with respect to the projections $\pi$ given by the splitting $G=N \rtimes H.$ On the other hand, as proved in the last proposition, when $N$ is normal, $$ \tilde \pi_H = \pi_H.$$ 
 \end{rem}

 \begin{rem} Let $\alpha  \geq 0.$  Then, $C^\ell_{N,H} ( \alpha ) =(C^r_{N,H} ( \alpha))^{-1}.$ 
Indeed,
\begin{equation*}
\begin{aligned}
nh \in C^\ell_{N,H} ( \alpha )\, & \Longleftrightarrow \, d(1,n) \leq \alpha d(1,h)\, \Longleftrightarrow  \, d(1,n^{-1}) \leq \alpha d(1,h^{-1} )\\
 &  \Longleftrightarrow  \, h^{-1} n^{-1} \in C^r_{N,H} ( \alpha)\,  \Longleftrightarrow \,  (nh)^{-1} \in C^r_{N,H} ( \alpha).
\end{aligned}
\end{equation*}
 \end{rem}

  \subsection{1-codimensional intrinsically Lipschitz maps}\label{1-codimensional intrinsically Lipschitz maps}

Let $G=N \rtimes H$ be a metric Lie group with $H$ 1-dimensional. Then there is $V \in \mathfrak g$ such that $H=\{ \exp (tV) \,: \, t\in \R \}.$

Denote by $S^+_G (N,H)$ and $S^-_G (N,H)$ the halfspaces
\begin{equation*}
\begin{aligned}
S^+_G (N,H) := \{ g \in G\, :\, \pi_H(g) = \exp (tV), \mbox{with } t\geq 0 \},\\
S^-_G (N,H) := \{ g \in G\, :\, \pi_H (g)= \exp (tV), \mbox{with } t\leq 0 \}.\\
\end{aligned}
\end{equation*}
Let $p\in N \rtimes H$ and $\alpha \geq 0$ and we consider the intrinsic cone $p\cdot C_{N,H} ( \alpha)$ with 1-dimensional axis $H$ as in Definition \ref{defiIntrCONE}. Then we denote
\begin{equation*}
\begin{aligned}
p\cdot C^+_{N,H} ( \alpha) := (p\cdot C_{N,H} ( \alpha)) \cap S^+_G (N,H),\\
p\cdot C^-_{N,H} ( \alpha) := (p\cdot  C_{N,H} ( \alpha) )\cap S^-_G (N,H).\\
\end{aligned}
\end{equation*}

We can characterize $H$-valued intrinsically Lipschitz functions using the fact that subgraphs and supergraphs contain half cones. Precisely, for  $\phi:N \to H$, with $\phi(n)= \exp (f(n)V)$ and $f:N \to \R$, we define the supergraph  $E^+_\phi$ and the subgraph $E^-_\phi$ of $\phi$ as
\begin{equation*}
\begin{aligned}
E^+_\phi := \{ n\exp (tV) \in G\, :\, n\in N, t >f(n) \},\\
E^-_\phi :=  \{ n\exp (tV) \in G\, :\, n\in N, t <f(n) \}.\\
\end{aligned}
\end{equation*}

Notice that if $\phi$ is a continuous map, then  
\begin{equation*}
\begin{aligned}
\overline{E^+_\phi }= \{ n\exp (tV) \, :\, n\in N, t \geq f(n) \},\quad \overline{E^-_\phi }=\{ n\exp (tV) \, :\, n\in N, t \leq f(n) \}
\end{aligned}
\end{equation*}
and 
\begin{equation*}
\begin{aligned}
\partial E^+_\phi = \partial E^-_\phi = \Gamma_\phi.
\end{aligned}
\end{equation*}

Moreover, any point $p\in \Gamma _\phi $ is both the limit of a sequence  $(p_h)_h \subset E^-_\phi $ and of a sequence $(q_h)_h \subset E^+_\phi .$ Indeed, if $p=n\phi (n)= n \cdot \exp(f(n)V),$ it is enough to choose
\begin{equation*}
\begin{aligned}
p_h = n\exp \left(\left(f(n) -\frac 1 h \right)V\right), \quad  \mbox{and} \quad q_h = n\exp \left(\left(f(n)+\frac 1 h \right)V \right).
\end{aligned}
\end{equation*}

We present a "sort" of right-invariant property of the intrinsic cones:
\begin{prop}\label{lemconi2}  Let $G=N \rtimes H$ be a metric Lie group  with $H$ 1-dimensional. Then for any $\alpha >0,$ it holds
\begin{equation*}
\begin{aligned}
ph\cdot C^+_{N,H} ( \alpha) \subset p\cdot C^+_{N,H} ( \beta),\quad \forall p\in G, h= \exp(tV)\in H, \mbox{ with } t>0,\\
ph\cdot C^-_{N,H} ( \alpha)\subset p\cdot C^-_{N,H} ( \beta),\quad \forall p\in G, h= \exp(tV)\in H, \mbox{ with } t<0,\\
\end{aligned}
\end{equation*}
for $\beta \geq \alpha +2.$
\end{prop}

  \begin{proof}
Fix $\alpha >0.$ By left translation invariant and Remark \ref{27remsuiconi}, it is sufficient to show that
\begin{equation*}
\begin{aligned}
h\cdot C^+_{N,H} ( \alpha) \subset C^+_{N,H} ( \alpha +2),\quad  \mbox{ for all } h= \exp(tV)\in H, \mbox{ with } t>0.\\
\end{aligned}
\end{equation*}
Let $p= m\ell \in C^+_{N,H} ( \alpha),$ we want to prove that $hp\in C^+_{N,H} ( \alpha +2).$ 
 
 Using the fact that $N$ is normal, it follows that
 \begin{equation*}
\pi_N(hp)= C_h(m), \quad \pi_H(hp)= h\ell.
\end{equation*}
Moreover, by definition of $ C^+_{N,H} ( \alpha),$ we have that $d(1,m) \leq \alpha d(1,\ell)$ and so
 \begin{equation}\label{disealpha1}
d(1,C_h(m) ) \leq 2d(1,h) + d(1,m) \leq (2+\alpha) [d(1,h)+d(1,\ell)].
\end{equation}
Finally, observing that
\begin{equation*}
h\ell = \exp(tV) \exp(sV) = \exp((t+s)V),
\end{equation*}
with $s,t >0$ by hypothesis, we get that $d(1,h)+d(1,\ell) =d(1,h\ell).$ Putting together this fact and  \eqref{disealpha1}   we obtain the thesis.

  \end{proof}  

Now we are able to prove the main result of this paper:
 \begin{theorem}\label{prop1codim}
 Let $G=N \rtimes H$ be a metric group  with $H$ 1-dimensional and $\pi_H:G \to H$ Lipschitz. Let $\phi:N \to H$ be a continuous  map and $L>0$. Then the following statements are equivalent:
 \begin{enumerate}
 \item $\phi$ is intrinsically $L$-Lipschitz;
 \item for all $m\in N,$ it holds 
 \begin{equation}\label{conihalfspaces}
 m\phi (m)\cdot C^+_{N,H} ( 1/L) \subset \overline{E^+_\phi },\quad \mbox{ and } \quad m\phi (m)\cdot C^-_{N,H} ( 1/L) \subset \overline{E^-_\phi }.
\end{equation}

\end{enumerate}
  \end{theorem}

 \begin{proof}
$(1) \Rightarrow (2).$ By contradiction, we assume that $ m\phi (m)\cdot C^+_{N,H} ( 1/L) \nsubseteq \overline{E^+_\phi }.$ That means that there is $n\in N$ and $t\in \R$ such that $$ n \exp (tV) \in (m\phi (m)\cdot  C^+_{N,H} ( 1/L)) \cap E^-_\phi . $$ Now, by $n \exp (tV) \in m\phi (m)\cdot  C^+_{N,H} ( 1/L)$ and notice that $d(1, \exp(tV))= |t|,$ we have that $n \exp (sV) \in m\phi (m)\cdot  C^+_{N,H} ( 1/L)$ for any $s\geq t$ and, by $n \exp (tV) \in  E^-_\phi ,$ we get that $t< f(n).$ As a consequence, for $s=f(n)>t$ we obtain a contradiction because
\begin{equation*}
 n \exp (f(n)V) \in (m\phi (m)\cdot  C^+_{N,H} ( 1/L)) \cap \Gamma_\phi \subset (m\phi (m)\cdot   C_{N,H} ( 1/L)) \cap \Gamma_\phi =\{m \phi (m)\},
\end{equation*}
where in the last equality we used Corollary \ref{coroll28dic}.

$(2) \Rightarrow (1).$ For all $0< \alpha <1/L$, it follows that
 \begin{equation*}\label{equ8}
 \begin{aligned}
m\phi (m)\cdot C_{N,H} ( \alpha) & = (m\phi (m)\cdot C^+_{N,H} ( \alpha)) \cup (m\phi (m)\cdot C^-_{N,H} ( \alpha))\\
& \subset E^+_\phi \cup E^-_\phi \cup \{m \phi (m)\}
\end{aligned}
\end{equation*}
 and, consequently, $m\phi (m)\cdot C_{N,H} ( \alpha) \cap \Gamma_\phi =\{m \phi (m)\}.$ Hence, by Corollary \ref{coroll28dic}, we obtain the thesis.
  \end{proof}

\section{Intrinsically Lipschitz maps: equivalent analytic conditions}\label{Some equivalent conditions of intrinsically Lipschitz maps} In this section, we give some equivalent conditions of intrinsically Lipschitz maps in the context of metric groups with semi-direct splitting. More precisely, the main result is Proposition \ref{propLip} which follows from the following statement:

  \begin{prop}\label{propLip29dic} Let $ (N \rtimes H,d)$ be a metric  group.   Let $\phi:N \to H$, $m \in N$ and $p =m \phi(m).$ Then the following statements are equivalent:
 \begin{enumerate}
 \item it holds \begin{equation*}\label{equ0.1}
d(1,  \phi_{p^{-1}} (n)) \leq L d(1, n), \quad \forall n\in E_{p^{-1}},
\end{equation*}
where  the map
  $\phi_q : E _q  \to H$ is defined as \eqref{defintraslfunct};
 \item it holds
 \begin{equation*}\label{equ0.2}
d(\phi(m), \phi (n)) \leq  L d(1,\pi_N (p^{-1} q )), \quad \forall n\in N \,\, \mbox{with } q= n \phi(n) \in \Gamma  _{\phi}.
\end{equation*}
\item   it holds
 \begin{equation*}
d(\phi (\pi_N (p) ), \phi (\pi_N (pn) ))   \leq  L d(1,n), \quad \forall n\in N.
\end{equation*}
\item  there is $\tilde L>0$ such that
\begin{equation*}\label{equ0.1}
d(1, q) \leq \tilde L d(1, \pi_N (q)), \quad \forall q \in \Gamma _{\phi_{p^{-1}}}. 
\end{equation*}
\item  there is $\bar L>0$ such that
\begin{equation*}
d( p ,q) \leq  \bar L d(1,\pi_N (p^{-1} q )),  \quad \forall q \in \Gamma  _{\phi}.
\end{equation*}
\item  for all $\hat L\geq L,$ it holds $$p\cdot C_{N,H} (1/\hat L) \cap \Gamma_{\phi } =\emptyset .$$
\end{enumerate} 
\end{prop}

 \begin{proof} $(1) \Leftrightarrow (2).$ The algebraic expression of the translated function $\phi_{p^{-1}}$ is more explicit thanks to the fact that $N$ is normal. More precisely,
 \begin{equation}\label{lipDEFI11}
\phi_{p^{-1}}  (n_1)= (\pi_H (pn_1))^{-1} \phi( \pi_N (pn _1))= \phi(m)^{-1} \phi \left( m C_{\phi (m)}(n_1) \right), \quad \forall n_1\in N
\end{equation}
and so, if we put $n=m C_{\phi (m)}(n_1)$ and observing that $\pi_N (p^{-1} q )= \pi_N (p^{-1} n)$, we obtain the equivalence between $(1)$ and $(2)$.

 $(1) \Leftrightarrow (3).$ Since $N$ is a normal subgroup, it follows $\pi_H(m\phi (m))=\pi_H (m\phi (m)n) = \phi (m),$ for all $n\in N.$ Therefore, by left invariance of $d$ and $\phi_{p^{-1}}(1)=1$ we have that
 \begin{equation*}
 \begin{aligned}
d(\phi (\pi_N (p) ), \phi (\pi_N (pn) ))  & =  d((\pi_H(p))^{-1} \phi (m),(\pi_H(pn))^{-1} \phi (\pi_N(pn))) = d(1,  \phi_{p^{-1}} (n)), 
\end{aligned}\end{equation*}
and so the equivalence of this two statements is true.
 
$(1) \Leftrightarrow (4).$  The equivalence follows immediately from triangle inequality. 

$(2) \Leftrightarrow (5).$ The implication $(2) \Rightarrow (5)$ follows from  the left invariant property of $d$ and  triangular inequality; indeed, recall that $\pi_N (p^{-1} q )=\phi(m)^{-1} m^{-1} n \phi (m)=C_{\phi (m)^{-1}}(m^{-1} n)$
\begin{equation*}
\begin{aligned}
d(n  \phi(n), m \phi(m)) & = d( \phi(n), n^{-1} m \phi (m ))\\
& = d( \phi(m)^{-1} \phi(n), C_{\phi (m)^{-1}}(n^{-1} m))\\
 & \leq  d( \phi(m), \phi(n)) + d(1, C_{\phi (m)^{-1}}(m^{-1} n)) \\
  & \leq  (1+ L) d(1,C_{\phi (m)^{-1}}(m^{-1} n)),\\
\end{aligned}
\end{equation*}
for every  $n\in N$. On the other hand, the implication $(5) \Rightarrow (2)$ holds because
\begin{equation*}
\begin{aligned}
d(  \phi (n),  \phi(m)) & = d(  n\phi (n), n \phi(m))\\
& \leq d(  n\phi(n), m\phi(m))+ d( m\phi (m),  n\phi(m))\\
& = d(  n\phi (n), m\phi(m))+ d( C_{\phi (m)^{-1}}(n^{-1} m) , C_{\phi (m)^{-1}}(n^{-1} n))\\
 & \leq  (1+ \bar L) d(1, C_{\phi (m)^{-1}}(m^{-1} n)),\\
\end{aligned}
\end{equation*}
for every  $n\in N,$ as desired.

$(1) \Leftrightarrow (6).$ The equivalence follows observing that 
 \begin{equation*}
p\cdot C_{N,H} (1/ \hat L) \cap \Gamma_{\phi } =\{p\} \quad \Leftrightarrow \quad C_{N,H} ( 1/\hat L) \cap \Gamma _{\phi_{p^{-1}}}=\{1\}.
\end{equation*}
where  $\phi_{p^{-1}}$ is defined as in \eqref{defintraslfunct}. Indeed, by left invariant property 
\begin{equation*}
L_{p^{-1}} \left( p \cdot C_{N,H} ( 1/\hat L) \cap \Gamma_{\phi }  \right) = C_{N,H} (1/\hat L) \cap \Gamma _{\phi_{p^{-1}}}.
\end{equation*}
 \end{proof} 
 
   \begin{prop}\label{propLip} Let $ (N \rtimes H,d)$ be a metric  group  such that $\pi_{N} $ is $k$-Lipschitz at $1$.   Let $\phi:N \to H$, $m \in N$ and $p =m \phi(m).$ Then the following statements are equivalent:
 \begin{enumerate}
  \item $\phi$ is intrinsically $L$-Lipschitz at point $n \in N$ with respect to  $d$ and with constant $L>0;$
 \item it holds one of the inequality in Proposition \ref{propLip29dic}.
 \end{enumerate}
\end{prop}

 \begin{proof}
	It is enough to combine Proposition \ref{coroll28dic}  and Proposition \ref{propLip29dic}.  
\end{proof}

The following result gives a relationship between intrinsically Lipschitz maps and the Lipschitz property of $\pi_H.$
\begin{prop}
Let $ (N \rtimes H,d)$ be a metric  group and let $\alpha \in (0,1) .$ Assume also that $\phi :N \to H$ is an intrinsically Lipschitz map with intrinsically Lipschitz constant not larger than $\alpha.$ Then, for any fixed $q\in \Gamma _\phi$ the projection $\pi_H |_{\Gamma _{\phi_{q^{-1}}}  \cap B(1,r)}$ is a $\frac{\alpha}{1-\alpha}$-Lipschitz map.  
\end{prop}

\begin{proof}
Fix $q\in \Gamma _\phi .$ We would like to show that
\begin{equation}\label{disconalpha}
d(\pi_H (p), \pi _H(g)) \leq \frac{\alpha}{1-\alpha}  d(p,g), \quad \mbox{ for all } p,g \in \Gamma _{\phi_{q^{-1}}} \cap B(1,r).
\end{equation}

By Proposition \ref{Defi splitting is locally Lipschitz} (4), we can prove \eqref{disconalpha} with $g=1.$ Hence 
\begin{equation*}
\begin{aligned}
d(1, \pi_H(p))&  = d(1,\phi_{q^{-1}} (\pi_N(p) )) \leq \alpha d(1, \pi_N(p)) \leq \alpha (d(1, p) + d(p, \pi_N(p))) \\
&  
 \leq  \alpha (d(1, p) +  d(1, \pi_H(p))),
\end{aligned}
\end{equation*}
which gives \eqref{disconalpha}, as desired.
\end{proof}

We conclude this section noting that, as in Euclidean setting, pointwise limits of intrinsic Lipschitz functions
are intrinsic Lipschitz. 
 
  \begin{prop}\label{ascoli} Let $ (N \rtimes H,d)$ be a metric  group.   Let $\phi_h:N \to H$ be intrinsically $L$-Lipschitz  for $h\in \N$ such that
  \begin{equation*}
\lim_{h\to \infty} \phi_h (m) = \phi (m),
\end{equation*}
for all $m\in N$ with $\phi : N\to H.$ Then  $\phi$ is intrinsic L-Lipschitz.
\end{prop}

  \begin{proof} 
The statement follows from  the following computation
 \begin{equation*}
\begin{aligned}
d(  \phi (n),  \phi(m)) & \leq d(  \phi (n),  \phi_h(n))  + d( \phi_h(n), \phi_h(m)) + d(  \phi_h(m),  \phi(m)) \\
& \leq 2\epsilon + L  d(1, C_{\phi _h (m)^{-1}}(m^{-1} n)) \\
& \leq 2\epsilon + 2L d(  \phi (m),  \phi_h(m))+ L  d(1, C_{\phi  (m)^{-1}}(m^{-1} n)) \\
& \leq  (2+2L)\epsilon + L  d(1, C_{\phi  (m)^{-1}}(m^{-1} n)).\\
\end{aligned}
\end{equation*}
  \end{proof}

\section{Intrinsically Lipschitz vs. metric Lipschitz functions}
 It is well know that intrinsically Lipschitz maps are not metric Lipschitz maps and viceversa. In this section we present some particular case when there is a link between these two notions. In particular, the main result is Proposition \ref{propLipvs.intr}.

\subsection{$d_\phi$ quasi-distance}\label{1quasi-distance}
We fix a metric  group $ (N \rtimes H,d)$ with semidirect structure given by subgroups $N $ and $H$ with $N$ normal.
We consider the projections:
$$\pi_N  : N \rtimes H \to N \qquad \text{ and } \qquad 
\pi_H  : N \rtimes H \to H.$$
Given a function $\phi:E \subset N \to H $, we define the function $d_\phi: E \times E \to \R^+$  as
\begin{equation}\label{defidf0}
d_\phi(n_1, n_2):= \frac 1 2 \left( d(1, \pi_ N (q_1^{-1} q_2)) + d(1, \pi_ N (q_2^{-1} q_1)) \right), \quad \mbox{for all } n_1, n_2 \in E,
\end{equation}
where $q_i:= n_i \phi(n_i) $ for $i=1,2.$
Notice that the points $q_i$ are arbitrary elements of the graph $ \Gamma _\phi$ of $\phi$ (see \eqref{Phi}).

 \begin{prop}\label{propQuasidistance}
Let $ (N \rtimes H,d)$ as above and let $\phi:E \subset N \to H $ be a function.
Assume that $\phi$ is locally  intrinsically  $L$-Lipschitz  and that $\pi _H : G \to H$ is a  $C$-Lipschitz map. 
Then the map $d_\phi$, as in \eqref{defidf0}, is a quasi-distance on every relatively compact subset of $E$.  
  \end{prop}
 
\begin{proof}
It is easy to see that $d_\phi$ is symmetric and $n_1=n_2$ yields $d_\phi(n_1,n_1)=0.$ Hence, we just need to check the weaker triangular inequality, i.e., 
\begin{equation}\label{weaktriangineq}
d_\phi(n_1,n_2) \leq  C (1+L) \left(d_\phi(n_1,n_3)  + d_\phi(n_3,n_2) \right),
\end{equation}
for all $n_1,n_2,n_3 \in E' \Subset E.$

Fix $E' \Subset E$ and let $n_1,n_2,n_3 \in E'$ such that $q_i= n_i \phi(n_i) \in \Gamma _\phi$ for $i=1,2,3.$ Using the Lipschitz property of $\pi _H$ (see Proposition $\ref{Defi splitting is locally Lipschitz}$ (3)) and the triangular inequality, we obtain that
 \begin{equation*}
\begin{aligned}
C^{-1} d( 1, \pi_N (q_1^{-1} q_2)) & \leq  d( 1, q_1^{-1} q_2)  \leq  d( q_1, q_3) +  d( q_3, q_2) \\
& \leq d( 1, \pi_N (q_1^{-1} q_3)) + d( 1, \pi_H (q_1^{-1} q_3))  + d( 1, \pi_N (q_3^{-1} q_2)) + d( 1, \pi_H (q_3^{-1} q_2)), 
\end{aligned}
\end{equation*} 
and so, since $\phi$ is an intrinsically Lipschitz map, it follows that
 \begin{equation*}
\begin{aligned}
C^{-1} d( 1, \pi_N (q_1^{-1} q_2)) & \leq  (1+L)\left(d( 1, \pi_N (q_1^{-1} q_3)) + d( 1, \pi_N (q_3^{-1} q_2)) \right).
\end{aligned}
\end{equation*} 
In a similar way, we conclude that
 \begin{equation*}
\begin{aligned}
C^{-1} d( 1, \pi_N (q_2^{-1} q_1)) & \leq  (1+L)\left(d( 1, \pi_N (q_2^{-1} q_3)) + d( 1, \pi_N (q_3^{-1} q_1)) \right),
\end{aligned}
\end{equation*} 
and, consequently, putting together the last two inequalities, \eqref{weaktriangineq} holds. 
\end{proof}

 \begin{prop}\label{propQuasidistance.2} Under the same assumptions of Proposition $\ref{propQuasidistance}$, we have that $d_\phi$ is equivalent to the metric $d$ restricted to the graph map $\Gamma _\phi$.
 \end{prop}
 
 \begin{proof} We would like to show that there are $c_1, c_2>0$ such that
 \begin{equation}\label{collegamento}
 c_1 d_\phi (n_1,n_2) \leq d(q_1, q_2) \leq c_2 \, d_\phi (n_1,n_2),
 \end{equation}
for every $n_1,n_2 \in E'\Subset E$ with $q_i= n_i \phi(n_i) \in \Gamma _\phi$ for $i=1,2.$
 
Fix $E'\Subset E$. Using the fact that the splitting is locally $C$-Lipschitz at $1$, we obtain that
 \begin{equation*}
\begin{aligned}
C^{-1} d( 1, \pi_N (q_1^{-1} q_2)) & \leq  d( 1, q_1^{-1} q_2), \quad \mbox{for all } n_1, n_2 \in E',
\end{aligned}
\end{equation*} 
where $q_i= n_i \phi(n_i) \in \Gamma _\phi$ for $i=1,2.$ Consequently, the left hand side of \eqref{collegamento} is satisfied with $c_1= 2C^{-1}.$

On the other side, by the intrinsically $L$-Lipschitz property of $\phi$ and Proposition \ref{propLip} (5), it follows that 
   \begin{equation*}
d(q_1,q_2) \leq (1+L) d(1, \pi _N (q_1^{-1} q_2)), \quad \mbox{for all } n_1, n_2 \in E',
\end{equation*}
where $q_i= n_i \phi(n_i) \in \Gamma _\phi$ for $i=1,2.$ Hence, the left hand side of \eqref{collegamento} is satisfied with $c_2=  L+1$ and  the proof is concluded. 
\end{proof}

\subsection{Intrinsically Lipschitz vs. metric Lipschitz functions}\label{2quasi-distance} It is a natural question to ask if intrinsically Lipschitz functions are metric Lipschitz functions provided that appropriate choices of the metrics in the domain or in the target spaces are made. The answer is almost always negative already in the particular case of the Carnot groups (see \cite[Remark 3.1.6]{FS16}, \cite[Example 3.24]{ArenaSerapioni}). 
However, something relevant can be stated in metric  groups:
\begin{prop}\label{prLipvs.intrinsicLip}
 Let  $ (N \rtimes H,d)$  be a metric  group and let $\phi: N \to H $ be an intrinsically Lipschitz function with graphing function $$\Phi: (N, d_\phi) \to (N \rtimes H,d),\quad \Phi (n):= n\phi(n), \forall n\in N,$$ where $d_\phi$ is defined as in \eqref{defidf0}.  If we also assume that  $\pi _H : N \rtimes H \to H$ is a locally Lipschitz map then, the graph map $\Phi$  is a metric Lipschitz function from  $(N, d_\phi)$ to $(N \rtimes H,d).$
  \end{prop}

\begin{proof}
	It is enough to combine Proposition \ref{propQuasidistance} and Proposition \ref{propQuasidistance.2}. 
\end{proof}
 
  \begin{prop}
Under the same assumptions of Proposition $\ref{prLipvs.intrinsicLip}$, it follows that $\phi$ is a metric Lipschitz function from  $(N, d_\phi)$ to $(H,d) .$
  \end{prop}

 \begin{proof}
 Notice that
 \begin{equation*} \begin{aligned}
\pi_N(\Phi (n)^{-1} \Phi (m)) & = \pi_N( \underbrace{  \phi (n)^{-1}n^{-1} m \phi (n)}_{ \in  N}  \underbrace{  \phi (n)^{-1} \phi (m)}_{ \in  H})= \phi (n)^{-1}n^{-1} m \phi (n),\\
\pi_H(\Phi (n)^{-1} \Phi (m)) & = \phi (n)^{-1} \phi (m),
\end{aligned} \end{equation*}
for any $n,m \in N$. Hence, by Proposition \ref{propLip} (2), we have that
 \begin{equation*} \begin{aligned}
d(\phi (n), \phi (m)) \leq L d(1, \phi (n)^{-1}n^{-1} m \phi (n)) \leq 2L d_\phi (n,m), \quad \forall n, m \in N,\\
\end{aligned} \end{equation*}
as desired.
\end{proof}

 We stress that in general it is impossible to find a $unique$ quasi distance independent of $\phi :M \to W$ working for all the intrinsic Lipschitz functions. On the other hand, this is true exactly when the codomain $W$ is a normal subgroup:
 
  \begin{prop}\label{propLipvs.intr}  Let  $ (M \ltimes W,d)$  be a metric  group and let $\phi: M \to W$ be a function. Then the following are equivalent:
\begin{enumerate}
\item $\phi$ is an intrinsically L-Lipschitz function;
\item the map graph $\Phi : (M, d) \to  (M \ltimes W,d)$  is a metric $\tilde L$-Lipschitz function.
\end{enumerate}
   \end{prop}
  
  \begin{proof}
	$(1) \Rightarrow (2).$ Fix $p=m\phi (m) \in M \ltimes W.$ The algebraic expression of the translated function $\phi_{p^{-1}}$ defined in \eqref{defintraslfunct} is more explicit thanks to the fact that $W$ is normal. More precisely, noting that
 \begin{equation*}
\pi_W(m\phi (m) a)= \pi_L( \underbrace{ ma }_{ \in  M}  \underbrace{  a^{-1}\phi (m)a}_{ \in  W})= C_{  a^{-1}} (\phi (m)), \quad \pi_M(m \phi (m) a)=ma, \quad \forall a \in M,
\end{equation*}
and so we have that	
	\begin{equation*}
\phi_{p^{-1}} (a) = C_{  a^{-1}} (\phi (m)^{-1}) \phi (ma), \quad \forall a \in M.
\end{equation*}
As a consequence, if we put $a=m^{-1} k \in M$ by the simply fact
\begin{equation*}
\begin{aligned}
 \Phi (m)^{-1} \Phi (k) = a a^{-1} \phi (m)^{-1} a \phi (ma) =a \phi_{p^{-1}} (a),
\end{aligned}
\end{equation*}
we obtain that
\begin{equation*}
\begin{aligned}
d(1, \Phi (m)^{-1} \Phi (k)) \leq d (1, a \phi_{p^{-1}} (a)) \leq (1+L) d (1, a) = (1+L) d(m,k),
\end{aligned}
\end{equation*}
as desired.

	$(2) \Rightarrow (1).$ Fix $p=m\phi (m) \in M\ltimes W.$  If we consider $a= m^{-1} k \in M,$ it follows that
\begin{equation*}
\begin{aligned}
d(1, \phi_{p^{-1}} (a) ) & = d (1, C_{  k^{-1} m} (\phi (m)^{-1}) \phi (k) )\\
& \leq d (1,  k^{-1} m) + d(1, \Phi (m)^{-1} \Phi (k))\\
& \leq (1+\tilde L) d (1, a),
\end{aligned}
\end{equation*}
i.e., by the arbitrariness of $k$, $\phi$ is intrinsically Lipschitz at point $m \in M$.

	\end{proof}

\begin{rem}
Proposition \ref{propLipvs.intr} could be false when $W$ is not normal subgroup. An example of this fact is shown in \cite{FS16} in the context of Carnot groups.
\end{rem}

\begin{rem}
Under the same assumptions of Proposition \ref{propLipvs.intr}, i.e. \emph{if \, $W$ is a normal subgroup,} the quasi distance $d_\phi $ defined as in \eqref{defidf0} does not depend of a map $\phi.$ Indeed, recall that $\pi_M$ is a homomorphism, then
\begin{equation*}
\begin{aligned}
\pi_M(\Phi (k)^{-1} \Phi (m))= k^{-1}m,
\end{aligned}
\end{equation*}
and so
\begin{equation*}
d_\phi (m,k) = d(m,k),  \quad \forall k,m  \in M.
\end{equation*}

\end{rem}

\section{Intrinsic graph as a subgroup} In this section, we present some explicit computations  about intrinsically Lipschitz graphs when they are subgroups of a metric group.  This section is inspired by the notion of intrinsic linear map in Carnot groups noting that here we don't have the homogeneous structure given by the intrinsic dilations.

\subsection{When $N$ is a normal subgroup}
 \begin{prop}\label{graphSubgroup}
Let $ (N \rtimes H,d)$ be a metric  group and let $\phi: N \to H$ such that  its graph $\Gamma_{\phi}$ is a subgroup of $G.$ Then, for any $n,m \in N$ it holds
\begin{enumerate}
\item $\Phi (n)^{-1} \Phi (m) =C_{\phi (n)^{-1}}(n^{-1} m) \phi (n)^{-1}\phi ( m);$
\item $\Phi (n) \Phi (m)^{-1} =nC_{\phi(n)\phi (m)^{-1}}(m^{-1}) \phi (n)\phi ( m)^{-1};$
\item $\Phi (n) \Phi (m)  = n C_{\phi (n)}( m) \phi (n) \phi (m) ;$
\item $(\Phi (n) \Phi (m))^{-1} =C_{\phi (m)^{-1}}(m^{-1} ) C_{(\phi (n)\phi (m))^{-1}}(n^{-1})  (\phi (n)\phi ( m))^{-1};$
\item $\phi (nm) = \phi(n) \phi ( C_{\phi (n)^{-1}}( m)) .$
\end{enumerate}
Moreover,
\begin{description}
\item[(a)] $\phi (C_{\phi (n)^{-1}}(n^{-1} m))= \phi (n)^{-1}\phi ( m);$
\item[(b)] $\phi (nC_{\phi(n)\phi (m)^{-1}}(m^{-1})) =  \phi (n)\phi ( m)^{-1};$
\item[(c)] $\phi (n C_{\phi (n)}( m)) =  \phi (n) \phi (m);$
\item[(d)] $\phi \left( C_{\phi (m)^{-1}}(m^{-1} ) C_{(\phi (n)\phi (m))^{-1}}(n^{-1}) \right) =   (\phi (n)\phi ( m))^{-1}.$
 \end{description}
\end{prop}

 \begin{proof} Since $\Gamma_{\phi}$ is a subgroup of $G,$ we have that for every $n,m \in N$
 \begin{equation*}
\Phi (n)^{-1} \Phi (m) = \Phi (k),
\end{equation*}
for some $k\in N$ and, consequently,  the equalities $(1)-(a)$ hold noting that
 \begin{equation*}
 \begin{aligned}
k& =\pi_N(\Phi (n)^{-1} \Phi (m))=\pi_N( \underbrace{ \phi (n)^{-1} n ^{-1} m \phi (n) }_{\in N} \underbrace{ \phi (n)^{-1} \phi (m) }_{\in H} )=C_{\phi (n)^{-1}}(n^{-1} m),\\
\phi (k) & = \phi (C_{\phi (n)^{-1}}(n^{-1} m))=\pi_H(\Phi (n)^{-1} \Phi (m)) = \phi (n)^{-1}\phi ( m).
\end{aligned}
\end{equation*}
In a similar way, it is possible to show the equalities $(2)-(3)-(4)$ and consequently $(b)$ and $(c).$

To prove the equality $(5),$ we observe that for any $n \in N$ and $h\in H$ there is a unique $m \in N$ such that
\begin{equation*}
n=\pi_N (hm).
\end{equation*}
More precisely, $m:=\pi _N (h^{-1}  n).$ Indeed,
 \begin{equation*}
 \begin{aligned}
\pi_N (h \pi _N (h^{-1}  n)) = \pi _N (h C_{h^{-1}}(n)) = C_{h}(C_{h^{-1}}(n))=n,
\end{aligned}
\end{equation*}
as desired. Moreover $m$ is unique because if
\begin{equation*}
\pi _N (h^{-1}  m_1) = \pi _N (h^{-1}  m_2) \end{equation*}
then, recall that $\pi _N (h^{-1}  m_1 hh^{-1} )=  C_{h^{-1}}(m_1)$, we get that $ C_{h^{-1}}(m_1)=C_{h^{-1}}(m_2)$ and so $m_1=m_2.$ Now, for any $n, k \in N$ if we put 
\begin{equation*}
m= \pi_N (\phi (n)k), \end{equation*}
by the equality $(c)$ it follows
 \begin{equation*}
 \begin{aligned}
\phi (nm) = \phi (n  \pi_N (\phi (n)k)) = \phi (n) \phi (k) = \phi (n) \phi (\pi_N( \phi (n)^{-1} m))= \phi(n) \phi ( C_{\phi (n)^{-1}}( m)),
\end{aligned}
\end{equation*}
i.e. $(5)$ is true and the proof is achieved.
 \end{proof}

 \begin{coroll}
Let $k\in\N.$ Under the same assumption of Proposition $\ref{graphSubgroup}$, if there is $C>0$ such that $$d(1, \phi (n)) \leq Cd(1,n^k), \quad \forall n\in N,$$ then $\phi$ is intrinsically $Ck$-Lipschitz.
 \end{coroll}
 
 \begin{proof}
	It is enough to combine Proposition \ref{graphSubgroup} (a) and Proposition \ref{propLip}  (2). 
\end{proof}

 \begin{coroll}
Let $k\in\N.$ Under the same assumption of Proposition $\ref{graphSubgroup}$, if there is $C>0$ such that $$d(1, \phi (n)) \leq Cd(1,n^k), \quad \forall n\in N,$$ then $\phi$ is intrinsically $Ck$-Lipschitz.
 \end{coroll}
 
 \begin{proof}
	It is enough to combine Proposition \ref{graphSubgroup} (a) and Proposition \ref{propLip}  (2). 
\end{proof}

\subsection{When $H$ is a normal subgroup}

 \begin{prop}\label{graphSubgroupHnormal}
Let $ (N \ltimes H,d)$ be a metric  group and let $\phi: N \to H$ such that its graph $\Gamma_{\phi}$ is a subgroup of $G.$ Then, for any $n,m \in N$ it holds
\begin{enumerate}
\item $\Phi (n)^{-1} \Phi (m) = n^{-1} m C_{m^{-1} n }(\phi (n)^{-1}) \phi (m );$
\item $\Phi (n) \Phi (m)^{-1} =n m ^{-1}C_{m}( \phi (n)\phi ( m)^{-1});$
\item $\Phi (n) \Phi (m)  = n m C_{m^{-1}}(  \phi (n)) \phi (m) ;$
\item $(\Phi (n) \Phi (m))^{-1} =(n m )^{-1} C_{nm}(  \phi (m)^{-1}) C_{n}(\phi (n)^{-1}).$
\end{enumerate}
Moreover,
\begin{description}
\item[(a)] $\phi (n^{-1} m)= C_{m^{-1} n }(\phi (n)^{-1}) \phi (m );$
\item[(b)] $\phi (n m ^{-1}) =  C_{m}( \phi (n)\phi ( m)^{-1});$
\item[(c)] $\phi (n  m) = C_{m^{-1}}(  \phi (n)) \phi (m);$
\item[(d)] $\phi \left( (n m )^{-1}\right) =  C_{nm}(  \phi (m)^{-1}) C_{n}(\phi (n)^{-1}).$
 \end{description}
\end{prop}

 \begin{proof} Since $\Gamma_{\phi}$ is a subgroup of $G,$ we have that for every $n,m \in N$
 \begin{equation*}
\Phi (n)^{-1} \Phi (m) = \Phi (k),
\end{equation*}
for some $k\in N$ and, consequently,  the equalities $(1)-(a)$ hold noting that
 \begin{equation*}
 \begin{aligned}
k& =\pi_N(\Phi (n)^{-1} \Phi (m))=\pi_N( \underbrace{  n ^{-1} m  }_{\in N} \underbrace{ m^{-1} n \phi (n)^{-1}  n ^{-1} m  \phi (m) }_{\in H} )=n ^{-1} m,\\
\phi (k) & = \phi (n ^{-1} m)=\pi_H(\Phi (n)^{-1} \Phi (m)) = C_{m^{-1} n }(\phi (n)) \phi (m ).
\end{aligned}
\end{equation*}
In a similar way, it is possible to show the other equalities $(2)-(3)-(4)$ and consequently $(b)-(c)-(d).$
\end{proof}

 \bibliographystyle{alpha}
\bibliography{DDLD}

\end{document}